\documentclass[12pt, a4paper, twoside]{amsart}

\usepackage[margin=2.85cm]{geometry}

\usepackage[T1]{fontenc}
\usepackage{amsmath}
\usepackage{amsthm}
\usepackage{amsfonts}
\usepackage{MnSymbol}
\usepackage{bbm}
\usepackage[UKenglish]{babel}
\usepackage{enumerate}
\usepackage[dvips]{graphicx}
\usepackage{bm}
\usepackage{ae}
\usepackage{hyperref}

\theoremstyle{definition}

\theoremstyle{remark}
\newtheorem*{rem}{Remarks}
\theoremstyle{plain}
\newtheorem{thm}{Theorem}
\newtheorem{lem}[thm]{Lemma}

\newcommand{\abs}[1]{\ensuremath{\left\vert #1 \right\vert}}
\DeclareMathOperator{\diam}{diam}
\DeclareMathOperator{\mat}{Mat}

\DeclareMathOperator{\bad}{Bad}
\DeclareMathOperator{\GL}{GL}
\newtheorem*{Dickinson}{Lemma (Dickinson)}

\newcommand{\mb}[1]{\mathbb{#1}}
\newcommand{\mbf}[1]{\mathbf{#1}}
\newcommand{\mcal}[1]{\mathcal{#1}}

\begin{document}

\title{Badly approximable systems of linear forms in absolute value}

\author{M. HUSSAIN}

\address{M. Hussain, Department of Mathematics and Statistics, La Trobe University, Melbourne, 3086, Victoria, Australia}

\email{M.Hussain@latrobe.edu.au}

\thanks{MH's visit to Aarhus was sponsored in part by La Trobe
  University's travel grant and from SK's Danish Research Council for
  Independent Research}

\author{S. KRISTENSEN}

\address{S. Kristensen, Department of Mathematical Sciences, Faculty
  of Science, University of Aarhus, Ny Munkegade 118,
  DK-8000 Aarhus C, Denmark}

\email{sik@imf.au.dk}

\thanks{SK's Research supported the Danish Research Council for
  Independent Research. }



\begin{abstract}
  In this paper we show that the set of mixed type badly approximable
  simultaneously small linear forms is of maximal dimension. As a
  consequence of this theorem we settle the conjecture stated in
  \cite{bad}.
\end{abstract}

\maketitle

\section{Introduction}
\label{sec:setup-result}

Let  $X =(x_{ij})\in \mb{R} ^{mn}$  be a point identified
as an $m\times n$ matrix. Throughout, $m$ and $n$ are fixed natural
numbers. Let
\begin{equation*}
q_{_{1}}x_{_{1i}}+q_{_{2}}x_{_{2i}}+...+q_{_{m}}x_{_{mi}}\text{ \ \
\ \ \ \ }(1\leq i\leq n),
\end{equation*}
be a system of $n$ linear forms in $m$ variables written more
concisely as $\mbf{q}X$. The classical result of Dirichlet \cite{dir}
states that for any point $X\in\mb{I}^{mn}$, there exist infinitely
many integer points $(\mbf{p, q})\in\mb{Z}^n\times \mb{Z}^m\setminus
\{\mbf 0\}$ such that
\begin{equation}\label{2}
  \|\mbf{q}X\|=|\mbf{q}X-\mbf{p}|:=\max_{1\leq i\leq
    n}|q_{_{1}}x_{_{1i}}+q_{_{2}}x_{_{2i}}+...+q_{_{m}}x_{_{mi}}-p_i|
  < |\mbf{q}|^{-\frac{m}{n}},
\end{equation}
where $\left\vert \mbf{q}\right\vert $ denotes the supremum norm i.e.
$\left\vert \mbf{q}\right\vert :=\max \left\{ \left\vert
    q_{_{1}}\right\vert ,\left\vert q_{_{2}}\right\vert ,\dots
  ,\left\vert q_{_{m}}\right\vert \right\}$, and $\| \mbf{q} X \|$
denotes the distance from $\mbf{q} X$ to the nearest vector with
integer coordinates in the supremum norm.

The right hand side of (\ref{2}) may be sharpened by a constant
$c(m,n)$ but the best permissible values for $c(m,n)$ are unknown
except for $m=n=1.$ A point $X\in \mb{R}^{mn}$ is said to be
\emph{badly approximable} if the right hand side of (\ref{2}) cannot
be improved by an arbitrary positive constant. Denote the set of all
such points as \textbf{Bad}$(m,n)$; that is $X\in \textbf{Bad}(m,n)$
if there exists a constant $C(X)>0$ such that
$$\|\mbf{q}X\|>C(X)|\mbf{q}|^{-\frac{m}{n}}\;\;\;\text {for
  all}\;\;\mbf{q}\in\mb{Z}^m \setminus \{\mbf{0}\}.$$

The Khintchine--Groshev theorem \cite{Groshev, K} implies that
$\textbf{Bad}(m,n)$ is of $mn$--dimensional Lebesgue measure zero and
a result of Schmidt \cite{schmidt} states that
$\dim\textbf{Bad}(m,n)=mn$, where $\dim A$ denotes the Hausdorff
dimension of the set $A$ -- see \cite{Falc2} for the definition of
Hausdorff dimension.

An alternative set to consider is obtained by replacing the distance
to the nearest integer vector $\| \cdot \|$ by the usual supremum
norm.  The Hausdorff dimension of the set
\[\textbf{Bad}^*(m,n)=\{X\in
\mb{I}^{mn}:|\mbf{q}X|>C(X)|\mbf{q}|^{-\frac{m}{n}+1}\;\;\;\text{for
  all}\;\;\mbf{q}\in\mb{Z}^m \setminus \{\mbf{0}\}\}\] is discussed in
\cite{bad} where it is proved that $\textbf{Bad}^*(m,1)=m.$

It can be easily seen that the above set is obtained from
$\textbf{Bad}(m,n)$ by simply imposing the additional condition that
$\mbf{p}=\mbf{0}$. It is a natural question to ask what happens for a
less restrictive condition on the allowed set of vectors $\mbf{p}$.
Such a condition gives rise to a range of hybrids between the set
$\textbf{Bad}(m,n)$ and $\textbf{Bad}^*(m,n)$, where some coordinates
of the image are considered in absolute value and some in the distance
to nearest integer.

Let $A \subseteq \mb{Z}^n$ be a lattice of integer vectors in a linear
subspace of dimension $u$, where $u$ is an integer and $0\leq u \leq
n$.  A consequence of the Dirichlet type theorem established by
Dickinson in \cite{mixed} is the following statement.

\begin{Dickinson}
  Suppose that $m+u > n$. For each $X\in \mb{R}^{mn}$ there exist
  infinitely many non-zero integer vectors $(\mbf{p, q})\in A \times
  \mb{Z}^{m}$ such that
  $$|\mbf{q}X-\mbf{p}|< C
  |\mbf{q}|^{-\frac{m + u}{n}+1}.$$
\end{Dickinson}
In view of this lemma, it is natural to consider the following badly
approximable set.  Let $\bad_A(m,n)$ denote the set of $X\in
\mb{R}^{mn}$ for which there exists a constant $C(X)>0$ such that
\begin{equation}
  \label{def}
  |\mbf{q}X-\mbf{p}|>C(X)|\mbf{q}|^{-\frac{m+u}{n}+1}\;\;\;\forall\;\;(\mbf{p,
    q})\in A \times \mb{Z}^{m} \setminus \{\mbf{0}\}.
\end{equation}

The set $\bad_A(m,2)$ is related to an exceptional set associated with
the linearization of germs of complex analytic diffeomorphisms of
$\mb{C}^m$ near a fixed point and is the $m$-dimensional version of
the Schr\"{o}der's functional equation, see \cite{Arn, drv, drv2} for
further details.

In the present note, as a consequence of the following theorem we
prove a conjecture from \cite{bad}.

\begin{thm}
  \label{thm:bad}
  The Hausdorff dimension of $\bad_A(m,n)$ is maximal. If $m +u\leq n$,
  the Lebesgue measure of $\bad_A(m,n)$ is full.
\end{thm}

\begin{rem}
\end{rem}
\begin{itemize}
\item [(i)] Clearly when $u=0$, $\bad_A(m, n)$ is identified with the set
$\bad^*(m, n)$ of \cite{bad}. It not only settles the
conjecture but also disagrees with the `final comment' of the paper that
for $m<n$ the set $\bad^*(m, n)$ is zero dimensional.

\item [(ii)] As a consequence of Khintchine--Groshev type theorem
  established in \cite{mixedkg} that for $m+u>n$,
  \[|\bad_A(m, n)|_{mn}=0,\] where $| \, . \, |_k$ denotes
  $k$-dimensional Lebesgue measure.

\item[(iii)] As a consequence of our theorem, it is clear that for
  $m+u\leq n$ the complementary set to $\bad_A(m.n)$ should be of zero
  $mn$ dimensional Lebesgue measure. This set corresponds to the set
  of well approximable systems of linear forms in the present setup,
  and the consequence stated here is indeed proved in \cite{mixedkg}.

\item [(iv)] An obvious corollary of the above theorem is that $\dim
  \bad_A(m,n)=mn.$

\end{itemize}

\section{Proof of Theorem \ref {thm:bad}}
\label{sec:proof-theorem}

Initially, we reduce the problem to the case when $A = \mb{Z}^u \times
\{(0,...,0)\}$. To see that it suffices to consider this case, let
$\{v_1, \dots, v_u\}$ be a basis of the lattice $A$, extend the basis
to a basis of $\mb{R}^n$ by adding vectors $\{v_{u+1}, \dots, v_n\}$ and
consider the map sending $v_i$ to $e_i$, the $i$'th standard basis
vector of $\mb{R}^n$. Evidently, this change of basis is a linear
automorphism of $\mb{R}^n$ and hence bi-Lipschitz. Since such maps
preserve Hausdorff dimension as well as the properties of being
null/full with respect to Lebesgue measure, it suffices to consider
the problem in the image of the map. The inequalities defining the
problem are also unchanged up to modifying the constant $C(X)$ by a
factor depending only on $A$.

With the particular case $A = \mb{Z}^u \times \{(0,...,0)\}$ in mind,
the defining inequalities take a particularly pleasing form. Namely,
$X \in \bad_A(m,n)$ if and only if
\begin{equation}
  \label{eq:3}
  \abs{(\mbf{p},\mbf{q})
  \begin{pmatrix}
    I_u & 0 \\
    X_u & \tilde{X}
  \end{pmatrix}}
  \geq C(X) \abs{\mbf{q}}^{-\frac{m+u}{n} + 1},
\end{equation}
for all $(\mbf{p},\mbf{q}) \in \mb{Z}^u \times \mb{Z}^m \setminus
\{\mbf{0}\}$, where $X$ is the matrix $(X_u, \  \tilde{X})$. Put
differently, we want an appropriate lower bound on the system of
linear forms
\begin{equation*}
  \max\left(\|\mbf{q\cdot
      x^{(1)}}\|,\ldots,\|\mbf{q\cdot x^{(u)}}\|, |\mbf{q\cdot
      x^{(u+1)}}|,\ldots,|\mbf{q\cdot x^{(n)}}|\right),
\end{equation*}
where $(\mbf{p},\mbf{q}) \in \mb{Z}^u \times \mb{Z}^m \setminus
\{\mbf{0}\}$.

We now split the proof into two parts depending upon the choices of
$m, u$ and $n$.

First, we discuss the case when $m+u=n$. It is easily seen that
\[
\mb{R}^{m(m+u)}\setminus \{X\in\mb{R}^{m(m+u)}:
\mbf{x}^{(u+1)},\ldots,\mbf{x}^{(n)} \text{ are linearly dependent}\}
\ = \ \bad_A(m,m+u).
\]
Since $\{X\in\mb{R}^{m(m+u)}: \mbf{x}^{(u+1)},\ldots,\mbf{x}^{(n)}
\text{ are linearly dependent}\}$ is a set of Hausdorff dimension less
then $m(m+u)$
\[
|\{X\in\mb{R}^{m(m+u)}: \mbf{x}^{(u+1)},\ldots,\mbf{x}^{(n)} \text{
  are linearly dependent}\}|_{m(m+u)}=0.
\]
 Hence,
the complement is full, \emph{i.e.} $\bad_A(m,m+u)$ is full. This
clearly implies that $\dim\bad_A(m,m+u) = m(m+u)$, which proves our
main theorem in the particular case $m+u=n$.

Note that in fact we get the stronger inequality
\begin{equation}
  \label{eq:2}
  |\mbf{q}X-\mbf{p}|>C(X)|\mbf{q}|\;\;\;\forall\;\;(\mbf{p,
    q})\in A \times \mb{Z}^{m} \setminus \{\mbf{0}\}
\end{equation}
in the set considered. This is much stronger than the defining
inequality of the set $\bad_A(m,n)$ and further underlines the
quantitative difference between the case $m+u \leq n$ and the converse
$m+u > n$.

Suppose now that $m + u < n$. We will argue much in the spirit of the
above. For a generic $X \in \mb{R}^{mn}$, the matrix $\tilde{X}$ in
\eqref{eq:3} has full rank. Performing Gauss elimination on the
columns of a matrix of the form of \eqref{eq:3} implies the existence
of an invertible $n \times n$-matrix $E(X)$ such that
\begin{equation*}
  \begin{pmatrix}
    I_u & 0 \\
    X_u & \tilde{X}
  \end{pmatrix} =
  \begin{pmatrix}
    I_u & 0 & 0 \\
    \hat{X} & I_m & 0
  \end{pmatrix} E(X).
\end{equation*}
Applying this matrix from the right to a vector $(\mbf{p}, \mbf{q})$,
we see that
\begin{equation*}
  (\mbf{p},\mbf{q})
  \begin{pmatrix}
    I_u & 0 \\
    X_u & \tilde{X}
  \end{pmatrix} =
  (\mbf{p},\mbf{q})
  \begin{pmatrix}
    I_u & 0 & 0 \\
    \hat{X} & I_m & 0
  \end{pmatrix} E(X) =
  \begin{pmatrix}
    \mbf{p} + \mbf{q}\hat{X} \\
    \mbf{q} \\
    \mbf{0}
  \end{pmatrix}^T E(X).
\end{equation*}

Multiplication by the matrix $E$ on the right hand side only serves to
distort the unit cube in the absolute value to a different
parallelipiped depending on $X$. This induces a different norm on the
image, but by equivalence of norms on Euclidean spaces, this
distortion can be absorbed in a positive constant.  In other words,
\begin{equation}
  \label{eq:1}
  \abs{\left(\mathbf{p, q}\right)\left(
      \begin{array}{cc}
        I_u & 0 \\
        X_u & \tilde{X}
      \end{array}
    \right)} \geq C(X)\abs{
    \begin{pmatrix}
      \mbf{p} + \mbf{q}\hat{X} \\
      \mbf{q} \\
      \mbf{0}
    \end{pmatrix}^T}
\end{equation}
for all $\left(\mathbf{p, q}\right)\in A \times \mb{Z}^{m} \setminus \{\mbf{0}\}.$

Finally, since the norm is the \emph{supremum} norm, we get
\begin{equation*}
  \abs{
    \begin{pmatrix}
      \mbf{p} + \mbf{q}\hat{X} \\
      \mbf{q} \\
      \mbf{0}
    \end{pmatrix}^T} \geq \abs{\mbf{q}}.
\end{equation*}
By \eqref{eq:1}, almost every $X$ is in $\bad_A(m,n)$, and in fact
with the stronger requirement from \eqref{eq:2}. This completes the
proof of Theorem \ref{thm:bad} in the case when $m+u \leq n$.

Now we consider the case $m+u>n.$ We will apply the following
extension of Lemma 2.1 of \cite{bad}.

\begin{lem}
  \label{lem:max-dim}
  Let $S \subseteq \mat_{(m +u-n)\times n}(\mathbb{R})$ of Hausdorff
  dimension $(m +u-n) n$. Let $\mcal{X} \subseteq  \GL_{n}(\mathbb{R})$
  be a set of positive $(n-u)n$--Lebesgue measure. Then, the set
  \begin{equation*}
    \Lambda = \left\{
       \begin{pmatrix}
        X^\prime \\
        \tilde{X}X^\prime
      \end{pmatrix}
      \in \mat_{(m+u) \times n}(\mathbb{R}) : X^\prime \in \mcal{X},
      \tilde{X} \in S \right\}
  \end{equation*}
  has Hausdorff dimension $mn$.
\end{lem}

\begin{proof}[Proof of Lemma \ref {lem:max-dim}]
  The proof is essentially an adaptation of the proof in \cite{bad}.
  As in that paper, the upper bound is trivial.

  Without loss of generality, we will assume that $|\mcal{X}|_{(n-u)n}
  < \infty$. If this is not the case, we will replace $\mcal{X}$ with
  a subset of $\mcal{X}$ of positive and finite measure.  Suppose now
  for a contradiction that $\dim \Lambda < mn$ and fix an $\epsilon >
  0$.  Then, there is a $\delta > 0$ and a cover $\mathcal{C}$ of
  $\Lambda$ by hypercubes in $\mathbb{R}^{(m+u)n}$ such that
  \begin{equation*}
    \sum_{C \in \mathcal{C}} \diam(C)^{mn-\delta} < \epsilon.
  \end{equation*}

  For a fixed $X^\prime \in \mcal{X}$, define the set
  \begin{equation*}
    B(X^\prime) = \left\{
      \begin{pmatrix}
        X^\prime \\
        \tilde{X}X^\prime
      \end{pmatrix}
      \in \mat_{(m+u) \times n}(\mathbb{R}) : \tilde{X} \in S \right\}.
  \end{equation*}
  Note that
  \begin{equation*}
    \mathcal{C}(X^\prime) =  \left\{ \left(
      \begin{pmatrix}
        X^\prime \\
        \mat_{(m +u-n) \times n} (\mathbb{R})
      \end{pmatrix} \cap C\right) \in \mat_{(m+u) \times
      n}(\mathbb{R}) : C \in \mathcal{C} \right\}
  \end{equation*}
  is a cover  of $B(X^\prime)$ by $(m +u-n)n$-dimensional hypercubes
  in the slice of the larger space $\mat_{m \times n}(\mathbb{R})$
  obtained by fixing the upper matrix to be $X'$.

  As in \cite{bad}, we define for each $C \in \mcal{C}$ a function,
  \begin{equation*}
    \lambda_C(X^\prime) =
    \begin{cases}
      1 & \text{if } \left(
        \begin{pmatrix}
          X^\prime \\
          \mat_{(m +u-n) \times n} (\mathbb{R})
        \end{pmatrix} \cap C\right) \neq \emptyset \\
      0 & \text{otherwise.}
    \end{cases}
  \end{equation*}
  It is easily seen that
  \begin{equation*}
    \int_{\mcal{X}} \lambda_C(X^\prime) dX^\prime \leq \diam(C)^{(n-u)n},
  \end{equation*}
  where the integral is with respect to the $(n-u)\times n$-dimensional
  Lebesgue measure. Also,
  \begin{equation*}
    \sum_{C \in \mathcal{C}(X^\prime)} \diam(C)^{(m +u-n)n - \delta} =
    \sum_{C \in \mathcal{C}} \lambda_C(X^\prime) \diam(C)^{(m +u-n)n - \delta}.
  \end{equation*}

  We integrate the latter expression with respect to $X^\prime$ to obtain
  \begin{multline*}
    \int_{\mcal{X}} \sum_{C \in \mathcal{C(X^\prime)}} \diam(C)^{(m
      +u-n)n - \delta} dX^\prime = \sum_{C \in
      \mathcal{C}}\int_{\mcal{X}} \lambda_C(X^\prime) dX^\prime
    \diam(C)^{(m +u-u)n -
      \delta} \\
    \leq \sum_{C \in \mathcal{C}} \diam(C)^{mn - \delta} < \epsilon.
  \end{multline*}
  Since the right hand side is an integral of a non-negative function
  over a set of positive measure, there must be an $X_0 \in \mcal{X}$ with
  \begin{equation*}
    \sum_{C \in \mathcal{C}(X_0)} \diam(C)^{(m +u-n)n - \delta} <
    \frac{\epsilon}{|\mcal{X}|_{(n-u)n}}.
  \end{equation*}
  Indeed, otherwise
  \begin{equation*}
    \int_{\mcal{X}} \sum_{C \in \mathcal{C}(X^\prime)} \diam(C)^{(m +u-n)n - \delta}
    dX^\prime
    \geq \int_{\mcal{X}} \frac{\epsilon}{|\mcal{X}|_{(n-u)n}} dX^\prime = \epsilon.
  \end{equation*}

  Take such an $X_0$. Since we have produced a cover of $B(X_0)$,
  whose $(m +u-n)n - \delta)$-length is less than $\epsilon$, it follows
  that $\dim B(X_0) \leq (m +u-n)n - \delta$.  The map from $B(X_0)$ to
  $S$ defined by
  \begin{equation*}
    \begin{pmatrix}
        X_0 \\
        \tilde{X}X_0
      \end{pmatrix} \mapsto
    \begin{pmatrix}
        I_{n\times n} \\
        \tilde{X}
      \end{pmatrix}
  \end{equation*}
  is evidently bi-Lipschitz as $X_0 \in \mcal{X}$ and hence invertible. It
  follows that $\dim S = \dim B(X_0) \leq (m +u-n)n - \delta$, which
  contradicts the original assumption and completes the proof.
\end{proof}

Proving the main theorem is now easy in the case $m +u> n$. Let
\begin{equation*}
  \mcal{X} = \left\{X \in \GL_{n}(\mathbb{R}): X =
    \begin{pmatrix}
      I_u & 0 \\
      X' & X''
    \end{pmatrix}\right\}.
\end{equation*}
Evidently, the $(n-u)n$-dimensional Lebesgue measure of $\mcal{X}$ is
positive. Let $S = \bad(m+u-n, n)$, the usual set of badly
approximable system of $n$ linear forms in $m+u-n$ variables.
Schmidt's theorem \cite{schmidt} tells us that $\dim S = (m+u-n)n$.

For $Y \in S$, there is a constant $C(Y) > 0$ such that for any
$\mathbf{r} \in \mathbb{Z}^{m+u-n}\setminus\{\mbf 0\}$ and any $\mathbf{p}
\in \mathbb{Z}^n$,
\begin{equation*}
  \abs{(\mathbf{p}, \mathbf{r})
    \begin{pmatrix}
      I_n \\
      Y
    \end{pmatrix}} \geq C(Y) \abs{\mathbf{r}}^{-\frac{m+u}{n} + 1}.
\end{equation*}
Multiplying a matrix $X \in \mcal{X}$ onto $\binom{I_n}{Y}$ only
changes the constant $C(Y)$ to another positive constant $C(X, Y,
\epsilon) > 0$. Hence,
\begin{equation}
  \label{eq:4}
  \abs{\mathbf{q}
    \begin{pmatrix}
      X \\
      YX
    \end{pmatrix}} \geq C(X, Y) \abs{\mathbf{q}}^{-\frac{m+u}{n} +
    1},
\end{equation}
for any $\mathbf{q} \in \mathbb{Z}^m \setminus \{\mbf 0\}$. The matrix
in \eqref{eq:4} has the form
\begin{equation*}
  \begin{pmatrix}
    I_u & 0 \\
    X_u & \tilde{X}
  \end{pmatrix}
\end{equation*}
by choice of $\mcal{X}$. In other words, using \eqref{eq:3} we see
that the set $\Lambda$ arising from Lemma \ref{lem:max-dim} is a
subset of $\bad_A(m,n)$ with some additional `dummy' coordinates
attached in the first $u$ rows.  It follows that $\dim \bad_A(m,n)
\geq \dim \Lambda = mn$, which completes the proof..

\noindent{\em Acknowledgements. } MH would like to thank Simon
Kristensen for his hospitality at Aarhus. 

\def\cprime{$'$}
\providecommand{\bysame}{\leavevmode\hbox to3em{\hrulefill}\thinspace}
\providecommand{\MR}{\relax\ifhmode\unskip\space\fi MR }
\providecommand{\MRhref}[2]{%
  \href{http://www.ams.org/mathscinet-getitem?mr=#1}{#2}
}
\providecommand{\href}[2]{#2}

\end{document}